\documentclass[11pt]{amsart}
  \usepackage[margin=3.8cm]{geometry}
\numberwithin{equation}{section}
\usepackage{amsmath, amsfonts,amsthm,amssymb,amscd, verbatim,graphicx,color,multirow,booktabs, caption,tikz,tikz-cd, mathdots,bm}
\usepackage{tikz-cd}
 \usepackage[pagebackref]{hyperref} 
\usetikzlibrary{positioning} 

\newtheorem{theorem}{Theorem}[section]
\newtheorem{lemma}[theorem]{Lemma}

\newtheorem{proposition}[theorem]{Proposition}
 \newtheorem{corollary}[theorem]{Corollary}
   \newtheorem{question}[theorem]{Question}

      \theoremstyle{definition}
     \newtheorem*{definition}{Definition}
     \newtheorem{example}[theorem]{Example}
     
     \theoremstyle{remark}
     \newtheorem{remark}[theorem]{Remark}

\newcommand{\Sym}{\mathop{\mathrm{Sym}}}
\newcommand{\Alt}{\mathop{\mathrm{Alt}}}

\newcommand{\Sch}{\mathop{\mathrm{Sch}}}
\newcommand{\Cay}{\mathop{\mathrm{Cay}}}

\newcommand{\SL}{\mathop{\mathrm{SL}}}

\newcommand{\Aff}{\mathop{\mathrm{Aff}}}
\newcommand{\PSL}{\mathop{\mathrm{PSL}}}

  \definecolor{mycolor}{rgb}{0.0,0.0,0.7}
  \definecolor{myred}{rgb}{0.75,0.0,0.16} 
  \definecolor{mygreen}{rgb}{0.0,0.4,0.16} 
  \definecolor{myviolet}{rgb}{1,0,1} 
   \definecolor{mypink}{rgb}{0.9,0,0.5}
	
	\newcommand{\MYhref}[3][black]{\href{#2}{\color{#1}{#3}}}%

\hypersetup{
colorlinks=true,
linkcolor=true,
linktocpage=true,
pageanchor=true,
hyperindex=true
}

 \AtBeginDocument{
     \hypersetup{
  linkcolor=mycolor,
  urlcolor=mypink,
citecolor=mygreen
}
     }

\makeatletter
\@namedef{subjclassname@2020}{%
  \textup{2020} Mathematics Subject Classification}
\makeatother


\author[Luca Sabatini]{Luca Sabatini}
\address{\parbox{\linewidth}{Luca Sabatini, Mathematics Institute, Zeeman Building, University of Warwick\\
Coventry CV4\,7AL, United Kingdom \vspace{0.1cm}}}
\email{luca.sabatini@warwick.ac.uk, sabatini.math@gmail.com} 

\begin{document} 
 \title[Groups that produce expanders]{Groups that produce expander graphs} 

\maketitle 

\begin{abstract} 
We survey the
known
group properties that a sequence of finite groups or group actions
needs to satisfy to admit subsets of bounded cardinality producing expander Cayley or Schreier graphs.
We prove that an infinite amenable group
and solvable groups of bounded derived length
do not produce expander Schreier graphs,
generalizing
with easier proofs
results of Lubotzky and Weiss for Cayley graphs.
In particular,
the poor expansion properties of a group action cannot in general be detected
by looking at the abelian sections or at the representations above the stabilizer of a point.
\end{abstract}

\vspace{0.5cm}
\section{Introduction} 

Expander graphs are sequences of finite graphs that are at the same time sparse and highly connected,
and play a basic role in various areas of mathematics and computer science \cite{HLW06,Lub12}.
Their existence follows by the probabilistic method,
but explicit constructions are much more difficult,
and many known examples involve a sequence of finite groups and their Cayley graphs.

Let $(G_n)_{n \geq 1}$ be a sequence of finite groups.\footnote{
If not explicitly stated otherwise,
when $(X_n)_{n \geq 1}$ is a sequence of finite graphs or groups or sets,
we will always assume that $|X_n| \to \infty$ when $n \to \infty$.
}
In 1993, Lubotzky and Weiss \cite{LW93} studied the group properties that $(G_n)_{n \geq 1}$ needs to satisfy to admit
a sequence of symmetric generating sets $(S_n)_{n \geq 1}$ of bounded cardinality
such that $(\Cay(G_n,S_n))_{n \geq 1}$ are expanders.
They famously asked if being an expander Cayley graph is a {\itshape group property},
namely whether $(\Cay(G_n,S_n))_{n \geq 1}$ being expanders implies the same for $(\Cay(G_n,T_n))_{n \geq 1}$,
where $(T_n)_{n \geq 1}$ is any other sequence of generating sets of bounded cardinality.
The answer turned out to be negative \cite{ALW01},
but it still makes sense to consider the original ``best case'' situation,
which is a group property by definition:

\begin{definition}
Let $(G_n)_{n \geq 1}$ be a sequence of finite groups.
We say the $(G_n)_{n \geq 1}$ is an {\itshape expanding sequence} (of groups)
if there exists a sequence $(S_n)_{n \geq 1}$ of symmetric generating sets of bounded cardinality
such that
$$ ( \, \Cay(G_n,S_n) \, )_{n \geq 1} $$
are expander graphs.
\end{definition}

With this terminology, the monumental work in \cite{KLN06,BGT11}
says that the sequence of all nonabelian finite simple groups is expanding.
For comparison, the ``random case'' is studied in several important papers, most notably \cite{BG08,BGGT15}.
Here a crucial role is played by the so-called {\itshape quasirandom groups},
introduced by Gowers \cite{Gow08} and recently characterized in group theoretical terms
by M. Barbieri and the author \cite{BS25}.
Finally, in the ``worse case'' scenario, the main known result remains \cite{BG10}
which shows that $(\SL_2(\mathbb{F}_p))_{p \in \mathcal{P}}$ is uniformly expanding
for any sequence of generating sets.\footnote{
Here $\mathcal{P}$ is an infinite set of primes.
Very recently, O. Becker and E. Breuillard have announced a generalization of this result, see
\texttt{\MYhref[mypink]{https://sites.google.com/view/orenbecker/upcoming-preprints?authuser=0\#h.ym9dnhs7xopf}{sites.google.com/view/orenbecker/upcoming-preprints}}.
 }\\

In this paper, we are interested in negative results.
We remark that there are very few available tools to prevent that a given sequence of finite groups is expanding -
informally, an expanding sequence has to be ``far from abelian'',
and the following question is open:

\begin{question} \label{quest1}
Let $(G_n)_{n \geq 1}$ be a sequence of finite groups
which can be generated by $d$ elements for some fixed $d \geq 2$
and satisfy the conditions in \cite{LW93} and \cite{HRV93}.
Is $(G_n)_{n \geq 1}$ an expanding sequence of groups?
\end{question}

The conditions above concern the algebraic structure of the groups,
controlling in particular the abelian sections and the degrees of the irreducible representations
(see Theorems \ref{thLWineq} and \ref{thHRV} below).
The answer to Question \ref{quest1} is most likely negative in general.
To build intuition in this direction, we will extend such ideas to the broader setting of {\itshape Schreier graphs},
and show the corresponding answer to be negative therein.
For a group $G$ acting transitively on a finite set $\Omega$, and $S = S^{-1} \subseteq G$,
the Schreier graph $\Sch(G \circlearrowright \Omega,S)$ is the graph with vertex set $\Omega$ and edges
$(\omega,\omega^s)$ for any $\omega \in \Omega$ and $s \in S$.

 \begin{figure}[h]
	\centering
	\includegraphics[width=5cm]{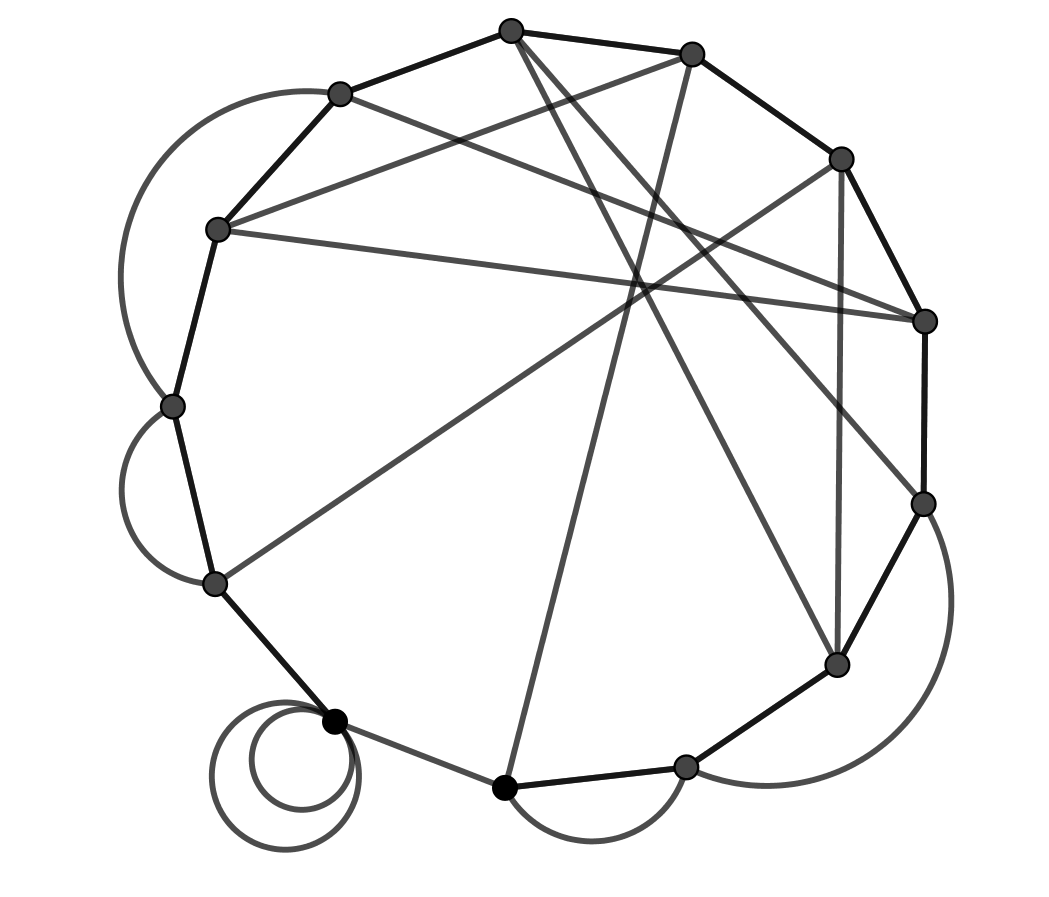}
	\caption{$\Sch( \, \Aff_1(13) \circlearrowright \mathbb{F}_{13} \, , \, \{ \pm 1, \cdot 2^{\pm 1} \} \, )$.}
	\end{figure}
	
	We would like to make an example.
Let $p$ be a prime and let $\Aff_1(p) \cong C_p \rtimes C_{p-1}$ be the affine group.
Then $\Aff_1(p)$ acts doubly transitively on $\mathbb{F}_p$,
and the stabilizer of a point is a maximal subgroup isomorphic to $C_{p-1}$.
There is no nontrivial section above the stabilizer,
and the double transitivity implies that the permutation representation splits nicely
into two irreducibles of degrees $1$ and $p-1$.
From these facts,
$(\Aff_1(p) \circlearrowright \mathbb{F}_p)_{p \geq 3}$
might appear
(as we will see, deceptively!)
to be a strong candidate to generate expander Schreier graphs.

\begin{definition}
Let $(G_n)_{n \geq 1}$ be a sequence of groups acting transitively on the finite sets $(\Omega_n)_{n \geq 1}$.
We say the $(G_n \circlearrowright \Omega_n)_{n \geq 1}$ is an {\itshape expanding sequence} (of group actions)
if there exists a sequence $(S_n)_{n \geq 1}$ of symmetric subsets of bounded cardinality such that
$$ ( \, \Sch(G_n \circlearrowright \Omega_n,S_n) \, )_{n \geq 1} $$
are expander graphs.
\end{definition}

\begin{remark}
Since a Schreier graph is a graph quotient of the corresponding Cayley graph,
it is easy to see that if $(G_n)_{n \geq 1}$ is an expanding sequence of groups,
then $(G_n \circlearrowright \Omega_n)_{n \geq 1}$ is always an expanding sequence of group actions.
\end{remark}

   Several theorems about expansion in Cayley graphs come
   from corresponding results on a fixed infinite group and a fixed generating set.
   In particular, \cite[Th. 3.1]{LW93} shows that an infinite amenable group does not produce
   expander Cayley graphs as finite quotients.
  In Section \ref{sec3}, we remark that a general version for actions actually holds:
   
   \begin{theorem} \label{thAmenableSG}
  Let $G$ be an infinite amenable group acting on the finite sets $(\Omega_n)_{n \geq 1}$.
  Fix a finite symmetric set $S \subseteq G$.
  Then $(\Sch(G \circlearrowright \Omega_n,S))_{n \geq 1}$ is not a sequence of expanders.
  \end{theorem}
 
 A modern introduction to amenable groups can be found in \cite{Bar18}.
 An important quality of \cite{Bar18} is that it treats {\itshape amenable actions} in a consistent fashion -
 in fact, the crucial reason behind the validity of Theorem \ref{thAmenableSG} is that amenability
 behaves well with respect to passing to actions \cite[Prop. 2.12]{Bar18}.
 This allows us to give a simpler proof of the original result for Cayley graphs,
 without using approximately invariant functions,
 which also holds for Schreier graphs.
 The obstacle
 (also present in \cite{LW93})
 that there is no infinite set in the hypotheses of Theorem \ref{thAmenableSG},
 will be overcome by considering the disjoint union of the $\Omega_n$'s and a simple numerical lemma.
 
 Using the amenability of the free solvable groups, we obtain:
 
\begin{theorem} \label{thMain}
  Let $(G_n)_{n \geq 1}$ be a sequence of finite solvable groups of bounded derived length.
  Then any sequence of actions $(G_n \circlearrowright \Omega_n)_{n \geq 1}$ is not expanding.
\end{theorem}

The Cayley graph case is \cite[Cor. 3.3]{LW93}.
Theorem \ref{thMain} implies that the sequence of group actions
$(\Aff_1(p) \circlearrowright \mathbb{F}_p)_{p \geq 3}$
is not expanding,
so the poor expansion properties of a group action cannot in general be detected
by looking at the abelian sections or at the representations above the stabilizer of a point.
In contrast we observe that \cite[Cor. 3.3]{LW93} can be recovered by looking at the abelian sections (Corollary \ref{corBDL}),
even if the original proof is based on \cite[Th. 3.1]{LW93}.


\begin{remark}
 The idea for Theorem \ref{thMain} originally came to the author from two exercises in Helfgott's notes \cite[Ex. 3.10-3.11]{Hel19},
  where it is shown that the Schreier graphs
  $( \Sch( \Aff_1(p) \circlearrowright \mathbb{F}_{p} , \{ \pm 1, \cdot \lambda^{\pm 1} \} ) )_{p \geq 3}$
  are not expanders.
  This is a particular instance of Theorem \ref{thMain},
  because $\Aff_1(p)$ is metabelian and $\{ \pm 1, \cdot \lambda^{\pm 1} \}$ is a specific generating set.
  \end{remark}

\vspace{0.1cm}
\section{Expanding sequences of groups and actions} \label{sec2}

In this section we survey the known group properties satisfied by an expanding sequence of groups or group actions.
We present the results in a streamlined fashion that emphasizes their asymptotic nature.
We also refer the reader to Chapter 3 of the unpublished book of Lubotzky and \.Zuk \cite{LZ03}.

\subsection{Isoperimetric number, expander graphs} \label{sec1.1}

Let $\Gamma$ be a finite undirected graph (loops and multiple edges are allowed).
Given a set of vertices $X \subseteq V(\Gamma)$, the {\itshape vertex boundary} $\partial_{ver} X$
is the set of vertices at distance precisely $1$ from $X$.
The {\itshape (vertex) isoperimetric number} is defined by
  $$ h_{ver}(\Gamma) \> := \> \min_{|X| \leq |\Gamma|/2} \frac{|\partial_{ver} X|}{|X|} . $$ 
 
 It is clear that $h_{ver}(\Gamma) >0$ if and only if $\Gamma$ is connected.
 
 \begin{definition}
 Let $(\Gamma_n)_{n \geq 1}$ be a sequence of regular graphs of bounded degree.
 Then $(\Gamma_n)_{n \geq 1}$ are {\itshape expander graphs} if $h_{ver}(\Gamma_n) \geq \varepsilon$
 for some fixed $\varepsilon >0$ and all $n$.
 \end{definition}
 
 As with many fundamental concepts in mathematics, there are several definitions of expanders:
 one can work with the {\itshape edge boundary} (and so with the {\itshape Cheeger constant}),
 or with the spectral gap of the adjacency matrix or of the Laplacian.
 All these notions are equivalent for graphs of bounded degree,
 and the convenience of the vertex boundary is that it has a nice interpretation when $\Gamma$ is a Schreier graph.
 
 If $\Gamma = \Sch(G \circlearrowright \Omega,S)$ and $X \subseteq \Omega$, then
 $$ \partial_{ver}X \> = \> XS \setminus X , $$
 where
$$ XS \> = \> \{ x^s : x \in X , s \in S \} . $$
We observe that $\partial_{ver}X = \emptyset$ if and only if $X$ is a union of orbits of $\langle S \rangle$.

\begin{remark}
 One obvious condition for a sequence of finite groups $(G_n)_{n \geq 1}$ to be expanding
 is that each $G_n$ can be generated by $d$ elements for some fixed $d \geq 2$.
 For a sequence of actions $(G_n \circlearrowright \Omega_n)_{n \geq 1}$,
 each $G_n$ needs to contain a {\itshape transitive subgroup} that can be generated by $d$ elements for some fixed $d$.
 \end{remark}

 \subsection{Abelian sections} 
 
For a group $G$, we write $G'=[G,G]$ to denote the commutator subgroup.
The quotient $G/G'$ is the largest abelian quotient of $G$.

\begin{theorem}[Th. 3.6 in \cite{LW93}] \label{thLWineq}
Let $(G_n)_{n \geq 1}$ be an expanding sequence of groups.
Then there exists a constant $c$ such that for every $n$ and $H \leqslant G_n$ we have
$$ |H:H'| \> \leq \> c^{\, |G_n:H|} . $$
\end{theorem} 


When $G$ is a (non necessarily finite) group, the function
$$ ab_k(G) \> := \> \sup_{|G:H| \leq k} |H:H'| $$
is the object of study of \cite{Sab23}, where it is named {\itshape abelianization growth}.
Theorem \ref{thLWineq} says that, if $(G_n)_{n \geq 1}$ is an expanding sequence of groups,
then $ab_k(G_n) \leq c^k$ for all $k$ and $n$,
for some constant $c$ possibly depending on the sequence.
The sequences of groups $(G_n)_{n \geq 1}$ such that
$ab_k( G_n )$ can be bounded by a function in $k$ (and not in $n$) are said to have the property
{\itshape BAb} ({\itshape Bounded Abelianizations}),
in analogy with the property {\itshape FAb} ({\itshape Finite Abelianizations}) which is interesting for infinite groups.
We now remark that solvable groups of bounded derived length do not even have BAb,
giving an alternative proof of the fact that they are not expanding \cite[Cor. 3.3]{LW93}.

\begin{lemma} \label{lemSolvBab}
Let $G$ be a finite solvable group of derived length $\ell \geq 1$.
Suppose that $ab_k(G) \leq f(k)$ for all $k$. 
Then $|G|$ can be bounded by a function in $\ell$ and $f$.
\end{lemma}
\begin{proof}
We work by induction on $\ell$.
If $\ell =1$, then $|G| = ab_1(G) \leq f(1)$.
Let $\ell \geq 2$ and let $G^{(i)}$ be the $i$-th term of the derived series of $G=G^{(0)}$.
Write $|G| = |G : G^{(\ell-1)}| |G^{(\ell-1)}|$.
Now $G^{(\ell-1)}$ is abelian, and so
$$ |G^{(\ell-1)}| \leq 
ab_{|G:G^{(\ell-1)}|} (G) \leq f(|G:G^{(\ell-1)}|) . $$
The group $G/G^{(\ell-1)}$ has derived length $\ell-1$ and satisfies
$$ ab_k(G/G^{(\ell-1)}) \leq ab_k(G) \leq f(k) . $$
By induction, $|G:G^{(\ell-1)}|$ can be bounded by a function in $\ell$ and $f$.
The proof follows.
\end{proof}

\begin{corollary} \label{corBDL}
Let $(G_n)_{n \geq 1}$ be a sequence of finite solvable groups of bounded derived length.
Then $(G_n)_{n \geq 1}$ does not have BAb, and in particular is not an expanding sequence.
\end{corollary}

\begin{remark}
In two parts of the
book \cite{KS11}
({\itshape Highly Speculative Conjecture 2} at page 117, and {\itshape Student Research Project 2} at page 119),
it is asked whether the fact that a sequence of finite groups $(G_n)_{n \geq 1}$
is nonexpanding can be purely detected by looking at the property BAb.
This is false, as there are sequences of groups with BAb but extremely fast abelianization growth \cite[Prop. 23]{Sab23}.
The typical example is $(V_n \rtimes \Sym(n))_{n \geq 1}$,
where $V_n \cong (C_{a_n})^{n-1}$ is a deleted permutation module
and $(a_n)_{n \geq 1}$ is a rapidly growing sequence of positive integers. 
\end{remark}

We now move to group actions.
The natural generalization of Theorem \ref{thLWineq} concerns the abelian sections above the stabilizer of a point:

\begin{theorem}[Th. 3 in \cite{Sab22}] \label{thSabIneq}
Let $(G_n \circlearrowright \Omega_n)_{n \geq 1}$ be an expanding sequence of group actions.
For each $n$, let $Y_n < G_n$ be the stabilizer of a point.
Then there exists a constant $c$ such that for every $n$ and $Y_n \leqslant H \leqslant G_n$ we have
$$ |H:H'Y_n| \> \leq \> c^{\, |G_n:H|} . $$
\end{theorem} 

The proof of this result is essentially as in Theorem \ref{thLWineq},
but some effort was made to polish the value of $c$.
In \cite{Sab22}, the author used Theorem \ref{thSabIneq} to show that
nilpotent groups of bounded nilpotency class do not produce expander Schreier graphs \cite[Th. 4]{Sab22}.
The key ingredient was the following inequality,
which has independent interest:

\begin{proposition}[Prop. 13 in \cite{Sab22}]
Let $G$ be a $d$-generated nilpotent group of nilpotency class $c$.
If $Y \leqslant G$, then
$$ |G:G'Y| \> \geq \> |G:Y|^{\varepsilon} , $$
where $\varepsilon >0$ only depends on $d$ and $c$.
\end{proposition}

This inequality fails dramatically in a non-nilpotent group:
if $Y<G$ is a maximal subgroup which is not normal, then $G=G'Y$.

 \subsection{Representation growth} 
 
 Let $Rep_k(G)$ denote the number of irreducible complex representations of $G$ of degree at most $k$.
 If $G$ can be generated by $d$ elements,
 then the two functions $ab_k(G)$ and $Rep_k(G)$ are connected by the inequalities
 \begin{equation} \label{eq2}
 Rep_k(G) \> \geq \> \frac{ab_k(G)}{k} 
 \end{equation}
 and
 \begin{equation} \label{eq1} 
 Rep_k(G) \> \leq \> 
 ( k! \cdot ab_{(k+1)!}(G) )^{O(d) \cdot (\log ab_{(k+1)!} (G) + k \log k)} . 
\end{equation} 
 (See \cite[Sec. 3]{Sab23}.)
 We point out that the bound in (\ref{eq2}) is the best possible.
	
	 \begin{example}
 Let $p$ be a prime and let $G = C_p \wr C_p$.
	It is a routine computation to show that $G$ has $p^2$ representations of degree $1$,
	and $p^{p-1}-1$ representations of degree $p$.
	It follows that
	$$ \frac{p \, Rep_p(G)}{ab_p(G)} = \frac{p (p^{p-1} +p^2 -1)}{p^p} = 1+ p^{3-p} - p^{1-p} . $$
	When $p$ grows, this quantity tends to $1$ from above.
	\end{example} 
 
 An expanding sequence $(G_n)_{n \geq 1}$ requires a bounded number of generators,
so it follows from (\ref{eq2}) and (\ref{eq1}) that $Rep_k(G_n)$ can be bounded by a function on $k$ for all $n$
if and only if $(G_n)_{n \geq 1}$ has BAb.
 The following theorem of de la Harpe, Robertson and Valette
 gives a stronger bound for expanding sequences.

\begin{theorem}[Prop. 4 in \cite{HRV93}] \label{thHRV}
Let $(G_n)_{n \geq 1}$ be an expanding sequence of groups.
Then there exists a constant $c$ such that for every $n$ we have
$$ Rep_k(G_n) \> \leq \> c^{\, k^2} $$
for all $k \geq 1$.
\end{theorem} 

   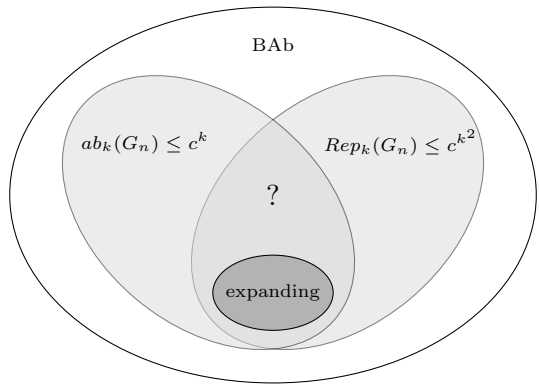
\begin{figure} 
 \centering
 \begin{tikzpicture}[scale=1][thick]
  \draw  (0,0.3) ellipse (3.5cm and 2.5cm);
    \draw[rotate=-50][fill=gray!30, opacity=0.5]  (0.5,0.7) ellipse (1.5cm and 2.2cm);
     \draw[rotate=50][fill=gray!30, opacity=0.5]  (-0.5,0.7) ellipse (1.5cm and 2.2cm);
  \draw [fill=gray!60] (0,-1) ellipse (0.8cm and 0.5cm);
   \node at (0,-1) {\tiny expanding};
    \node at (0,2.3) {\tiny BAb};
    \node at (-1.7,1) {\tiny $ab_k(G_n) \leq c^k$};
     \node at (1.7,1) {\tiny $Rep_k(G_n) \leq c^{k^2}$};
     \node at (0,0.3) {?};
    \end{tikzpicture} 
   \caption{The relations between the group properties BAb,
   expanding, and the conditions in \cite{LW93} and \cite{HRV93}.}
    \end{figure}

In the case of a group action $G \circlearrowright \Omega$,
we count the irreducible complex representations 
of degree at most $k$ appearing in the natural permutation representation.
Equivalently, we count the representations of degree at most $k$ lying above the stabilizer of a point.
We denote this quantity by $Rep_k(G \circlearrowright \Omega)$.

\begin{theorem}[Prop. 2.8 in \cite{Pis17}] \label{thPis}
Let $(G_n \circlearrowright \Omega_n)_{n \geq 1}$ be an expanding sequence of group actions.
Then there exists a constant $c$ such that for every $n$ we have
$$ Rep_k(G_n \circlearrowright \Omega_n) \> \leq \> c^{\, k^2} $$
for all $k \geq 1$.
\end{theorem} 

We now make a remark on the sharpness of these theorems.

The estimate in Theorem \ref{thLWineq} (and so in Theorem \ref{thSabIneq}) is the best possible.
Let $(H_n)_{n \geq 1}$ be an expanding sequence of groups and fix a prime $p$ not dividing $|H_n|$ for any $n$.
Let $G_n = \mathbb{F}_p[H_n] \rtimes H_n$.
By \cite[Th. 5.13]{LZ03} we have that $(G_n)_{n \geq 1}$ is expanding,
and, whenever $k=|H_n|$, we have $ab_k(G_n) = p^k$.

The estimate in Theorem \ref{thPis} is also the best possible \cite[Th. 2.11]{Pis17},
but the sharpness of Theorem \ref{thHRV} remains unknown (see Question \ref{questW}).

\vspace{0.1cm}
\section{Amenable actions and Theorems \ref{thAmenableSG} and \ref{thMain}} \label{sec3}

In this section we prove Theorems \ref{thAmenableSG} and \ref{thMain}.

 \subsection{Amenable groups and actions} 

As well as with expander graphs,
the notion of amenable group can be defined in several equivalent ways.
The following is a more general definition for group actions (see \cite[pag. 434]{Bar18}).

  \begin{definition}
 The action of a group $G$ on a set $\Omega$ is an {\itshape amenable action}
 if there exists a $G$-invariant mean on the subsets of $\Omega$.
  \end{definition} 
  
  Every action on a finite set is amenable, by taking the counting measure.
  A group is called an {\itshape amenable group} if its own action by multiplication is amenable.
  It is relevant to say that all abelian groups are amenable,
and that extensions of amenable groups are amenable.
In particular, all solvable groups are amenable.
 The following observation is crucial for our purposes:
 
  \begin{proposition} \label{propAmenable}
 If $G$ is amenable as a group, then every action of $G$ on every set is amenable.
  \end{proposition}
  \begin{proof}
  See \cite[Prop. 2.12]{Bar18}.
   \end{proof}

In the context of Cayley graphs,
it is useful to consider a combinatorial characterization of amenability which is due to F\o lner \cite{Fol55}.
When adapted to group actions, it takes the following form:
  
  \begin{lemma}[F\o lner condition for group actions] \label{lemFol}
   Let the action of the group $G$ on $\Omega$ be amenable.
   Then for every finite subset $S \subseteq G$
  there exists a sequence of finite subsets $(X_i)_{i \geq 1}$ of $\Omega$
  such that
  $$ \frac{|X_i S \setminus X_i|}{|X_i|} \> \to \> 0 $$
  when $i \to \infty$.
  \end{lemma}
  \begin{proof}
  See \cite[Th. 3.23 (1)$\Rightarrow$(5)]{Bar18}.
  \end{proof}
  
  The converse of Lemma \ref{lemFol} also holds, but we do not need it.
  The analogy between the F\o lner condition and the definition of isoperimetric number is clear.

  \subsection{F\o lner condition on the disjoint union} 
  
  In the hypotheses of Theorem \ref{thAmenableSG} we have a sequence of finite sets $(\Omega_n)_{n \geq 1}$,
  while the notion of amenable action is only relevant for infinite sets.
  In the Cayley graph case, Lubotzky and Weiss \cite[Th. 3.1]{LW93} overcome this apparent obstruction
  by defining suitable functions on the finite quotients of $G$,
  and then use properties of these functions and their norms.
  We will use a more direct approach, which immediately generalizes to actions:
  we act on the disjoint union of the $\Omega_n$'s and use the following lemma.
 
  \begin{lemma} \label{lemTec}
  Let $(m_i)_{i \geq 1}$ be positive integers, and let
  $$ (a_{i,j})_{i \geq 1, 1 \leq j \leq m_i}
  \hspace{1cm} \mbox{and} \hspace{1cm}
  (b_{i,j})_{i \geq 1, 1 \leq j \leq m_i} $$
  be non-negative integers, $b_{i,j} > 0$.
  Suppose that
  $$ \lim_{i \to \infty} \> \frac{\sum_{j=1}^{m_i} a_{i,j}}{\sum_{j=1}^{m_i} b_{i,j}} \> = \> 0 . $$
  Then there exists a sequence $(k_i)_{i \geq 1}$ such that
  $$ \lim_{i \to \infty} \> \frac{a_{i,k_i}}{b_{i,k_i}} \> = \> 0 . $$
  \end{lemma}
  \begin{proof}
   By contradiction, suppose that
   $$ \limsup_{i \to \infty} \> \min_{j} \left( \frac{a_{i,j}}{b_{i,j}} \right) \> = \> \delta >0 . $$
   This means that there exist infinitely many $i$ such that
   \begin{equation} \label{eqTec1}
   \frac{a_{i,j}}{b_{i,j}} \geq \frac{\delta}{2}
   \end{equation}
   holds for all $1 \leq j \leq m_i$.
   Fix $i$ such that (\ref{eqTec1}) is true,
   and let $x_j=a_{i,j}$ and $y_j=b_{i,j}$, $n=m_i$.
   It is easy to see that
   \begin{equation*}  \label{eqTec2}
   \min_{1 \leq j \leq n} \left( \frac{x_j}{y_j} \right) \> \leq \> \frac{\sum_{j=1}^n x_j}{\sum_{j=1}^n y_j} \> 
   \leq \>  \max_{1 \leq j \leq n} \left( \frac{x_j}{y_j} \right) . 
   \end{equation*}
   From the left side and (\ref{eqTec1}), we obtain that
  $$ \frac{\sum_{j = 1}^{m_i} a_{i,j}}{\sum_{j = 1}^{m_i} b_{i,j}} \> \geq \> \frac{\delta}{2} $$
   holds for infinitely many $i$.
   This gives the desired contradiction.
  \end{proof}
  
  \begin{proof}[Proof of Theorem \ref{thAmenableSG}]
  Let $G$ be an amenable group acting on $\Omega_n$ for every $n$,
  and let $S$ be a finite symmetric subset of $G$.
  Then $G$ acts naturally on the disjoint union $\Omega= \bigsqcup_{n \geq 1} \Omega_n$.
  By Proposition \ref{propAmenable}, the action of $G$ on $\Omega$ is amenable.
  Therefore, by Lemma \ref{lemFol} there exists a sequence of finite subsets $(X_i)_{i \geq 1}$ of $\Omega$,
  such that $|X_iS \setminus X_i|/|X_i| \to 0$ when $i \to \infty$.
  Since the $X_i$'s are finite,
  for each $i$ there exists $m_i$ such that $X_i \subseteq \bigsqcup_{i = 1}^{m_i} \Omega_i$.
  Moreover, observe that $\Omega_i S = \Omega_i$ for each $i$.
  Hence we can write
  $$ \frac{|X_iS \setminus X_i|}{|X_i|} \> = \> 
  \frac{\sum_{j=1}^{m_i} |(X_i \cap \Omega_j)S \setminus (X_i \cap \Omega_j)|}{\sum_{j=1}^{m_i} |X_i \cap \Omega_j|} . $$
   We apply Lemma \ref{lemTec} with
   $$ a_{i,j} = |(X_i \cap \Omega_j)S \setminus (X_i \cap \Omega_j)|
   \hspace{1cm} \mbox{and} \hspace{1cm}
   b_{i,j} =|X_i \cap \Omega_{j}| ,$$
   avoiding the $i,j$ where $b_{i,j}=0$.
   Let $(k_i)_{i \geq 1}$ be the sequence obtained in this way,
   and let $Y_i =X_i \cap \Omega_{k_i} \subseteq \Omega_{k_i}$.
  By the definition of isoperimetric number, we have
   $$ h_{ver}( \, \Sch(G \circlearrowright \Omega_{k_i},S) \, ) \leq \frac{|Y_iS \setminus Y_i|}{|Y_i|}
   \to 0 $$
   when $i \to \infty$.
   Thus $(\Sch(G \circlearrowright \Omega_n,S))_{n \geq 1}$ is not a sequence of expanders.
  \end{proof}

  \subsection{Solvable finite groups} 
  
  The proof of Theorem \ref{thMain} is as in \cite[Cor. 3.3]{LW93},
  just using Theorem \ref{thAmenableSG} as base stone.
  If $F_d$ denotes the free group on $d$ generators,
  then the {\itshape free solvable group} of derived length $\ell$ is defined by
  $$ L_{d,\ell} \> := \> F_d/F_d^{(\ell)} , $$
  where $F_d^{(\ell)}$ denotes the $\ell$-term in the derived series of $F_d = F_d^{(0)}$.
  It is clear that $L_{d,\ell}$ is solvable (of derived length $\ell$) and so amenable.
  Let $x_1,\ldots,x_d \in L_{d,\ell}$ be projections of the free generators.
  The universal property of $L_{d,\ell}$ is the following:
  for any solvable group $G=\langle g_1,\ldots,g_d \rangle$ of derived length at most $\ell$,
  there exists a homomorphism $\pi \colon L_{d,\ell} \to G$ such that $\pi(x_i) = g_i$ for all $i=1,\ldots,d$.
  
  \begin{proof}[Proof of Theorem \ref{thMain}]
  Let $(G_n)_{n \geq 1}$ be a sequence of solvable groups of derived length at most $\ell$,
  and let $(S_n)_{n \geq 1}$ be a sequence of symmetric subsets of cardinality at most $d$.
  Consider $L_{d,\ell}$,
  and let $S=\{ x_1,\ldots,x_d\} \subseteq L_{d,\ell}$ be a set of generators coming from the free generators of $F_d$.
  For each $n$ there exists $\pi_n \colon L_{d,\ell}  \to G_n$ such that $\pi_n(S)=S_n$.
  If $L_{d,\ell}$ acts on $\Omega_n$ with the action induced from the quotient $G_n$,
  we have
  $$ \Sch(G_n \circlearrowright \Omega_n,S_n) \> = \> \Sch(G_n \circlearrowright \Omega_n,\pi_n(S)) \> \cong \>
  \Sch(L_{d,\ell} \circlearrowright \Omega_n,S) . $$
  The proof follows from Theorem \ref{thAmenableSG}.
  \end{proof}

\section{Open problems} 

Question \ref{quest1} highlights the lack of tools preventing expansion in Cayley graphs.
Moreover, there are a few natural questions arising from Theorems \ref{thLWineq} and \ref{thHRV}.

\begin{question} \label{questRel}
Let $(G_n)_{n \geq 1}$ be a sequence of finite groups that are all $d$-generated for some fixed $d \geq 2$.
Explore the relation between the two properties $ab_k(G_n) \leq (c_1)^k$ and $Rep_k(G_n) \leq (c_2)^{k^2}$
(for possibly distinct constants $c_1$ and $c_2$).
Does one condition imply the other?
\end{question}

Question \ref{questRel} is related to the following:

\begin{question} \label{questRepAb}
Improve (\ref{eq1}), i.e. find a better upper bound for $Rep_k(G)$ in terms of $ab_k(G)$.
\end{question}

We also have

\begin{question}[Wigderson \cite{Pis17}] \label{questW}
Is the $c^{\, k^2}$ bound in Theorem \ref{thHRV} the best possible?
\end{question}

For what concerns specific examples,
it is plenty of sequences of finite groups for which it is not known whether they are expanding.
We say that a finite group is {\itshape anabelian} if it has only nonabelian composition factors.

\begin{question} \label{questAnab}
Let $(G_n)_{n \geq 1}$ be a sequence of anabelian groups that are all $d$-generated for some fixed $d \geq 2$.
Is $(G_n)_{n \geq 1}$ an expanding sequence of groups?
\end{question}

The abelianization growth and the representation growth of anabelian groups
are much slower than what is required by Theorems \ref{thLWineq} and \ref{thHRV}:
by \cite[Th. 3]{Sab23} we have $ab_k(G_n) \leq k^{O(1)}$ for all $n$,
while by \cite[Th. 4]{Sab23} we have $Rep_k(G_n) \leq (k!)^{O(1)}$ for all $n$.
In particular, a negative answer to Question \ref{questAnab} would also answer Question \ref{quest1}.
The following is a pretty example of anabelian sequence:

\begin{question} \label{questPSL}
Let $p$ be a prime and let $G_p = \PSL_2(\mathbb{F}_p)^p$ be the direct product of $p$ copies of $\PSL_2(\mathbb{F}_p)$.
Is $(G_p)_{p \geq 5}$ an expanding sequence of groups?
\end{question}

We observe that $G_p$ can be generated by a bounded number of elements for all $p$ \cite{Wie74}.
For comparison, Bartholdi and Kassabov \cite{BK25} studied sequences of groups of type 
$(\Alt(n)^{f(n)})_{n \geq 5}$, and proved that they are expanding if $f(n) \leq \exp((\log n)^{O(1)})$.


   \vspace{0.5cm}

\end{document}